\documentclass{amsart}
\usepackage{amssymb}

\newtheorem{theorem}{Theorem}[section]
\newtheorem{lemma}[theorem]{Lemma}
\newtheorem{corollary}[theorem]{Corollary}

\theoremstyle{definition}

\newtheorem*{acknowledge}{Acknowledgment}

\DeclareMathOperator{\hI}{\hat{\mathcal{I}}}

\theoremstyle{remark}

\numberwithin{equation}{section}

\newif\ifshowflag
\showflagtrue   
\newif\ifshowqstn
\showqstntrue   
\showqstnfalse 
\newif\ifshowinfo
\showinfotrue   
\showinfofalse 

\renewcommand{\phi}{\varphi}

\DeclareMathOperator{\diam}{Diam}

\catcode`\@=11       

\def\rf#1{\@rf{#1}#1:;;}
\def\rfs#1{\@rfs{#1}#1:;;}
\def\rfm#1{\@rfF#1<>;;}

\def\@C{C}\def\@CC{CC}\def\@E{E}\def\@F{F}\def\@L{L}\def\@P{P}\def\@Q{Q}
\def\@R{R}\def\@K{K}\def\@S{S}\def\@T{T}\def\@TT{TT}\def\@X{X}\def\@s{s}\def\@f{f}\def\@A{A}

\def\@rf#1#2:#3;;{\def\@b{#2}
  \ifx\@b\@C Corollary~\ref{#1}\else%
  \ifx\@b\@E (\ref{#1})\else
  \ifx\@b\@F Fact~\ref{#1}\else%
  \ifx\@b\@L Lemma~\ref{#1}\else%
  \ifx\@b\@P Proposition~\ref{#1}\else%
  \ifx\@b\@Q Question~\ref{#1}\else%
  \ifx\@b\@R Remark~\ref{#1}\else%
  \ifx\@b\@K Conjecture~\ref{#1}\else%
  \ifx\@b\@S Section~\ref{#1}\else%
  \ifx\@b\@A Appendix~\ref{#1}\else%
  \ifx\@b\@T Theorem~\ref{#1}\else%
  \ifx\@b\@TT Theorem~\ref{#1}\else%
  \ifx\@b\@X Example~\ref{#1}\else%
  \ifx\@b\@s Subsection~\ref{#1}\else
  \ifx\@b\@f Figure~\ref{#1}\else%
  \ref{#1}\fi\fi\fi\fi\fi\fi\fi\fi\fi\fi\fi\fi\fi\fi\fi}

\def\@rfs#1#2:#3;;{\def\@b{#2}
  \ifx\@b\@C Corollaries~\ref{#1}\else%
  \ifx\@b\@CC Corollaries~\ref{#1}\else%
  \ifx\@b\@F Facts~\ref{#1}\else%
  \ifx\@b\@L Lemmas~\ref{#1}\else%
  \ifx\@b\@P Propositions~\ref{#1}\else%
  \ifx\@b\@Q Questions~\ref{#1}\else%
  \ifx\@b\@R Remarks~\ref{#1}\else%
  \ifx\@b\@S Sections~\ref{#1}\else%
  \ifx\@b\@T Theorems~\ref{#1}\else%
  \ifx\@b\@TT Theorems~\ref{#1}\else%
  \ifx\@b\@X Examples~\ref{#1}\else%
  \ifx\@b\@s \S\ref{#1}\else
  \ifx\@b\@f Figures~\ref{#1}\else%
  \ref{#1}\fi\fi\fi\fi\fi\fi\fi\fi\fi\fi\fi\fi\fi}

\def\@rfF<#1>#2;;{\def\@c{#2}
  \@rfs{#1}#1:;;\ifx\@c\empty\else\@rfL:#2;;\fi}

\def\@rfL:#1<#2>#3;;{\def\@b{#2}\def\@c{#3}
  #1\ifx\@b\empty\else\ref{#2}\ifx\@c\empty\else\@rfL:#3;;\fi\fi}

\catcode`\@=12       

\def\vint_#1{\mathchoice
          {\mathop{\vrule width 6pt height 3 pt depth -2.5pt
                  \kern -8pt \intop}\nolimits_{\kern -4pt#1}}%
          {\mathop{\vrule width 5pt height 3 pt depth -2.6pt
                  \kern -6pt \intop}\nolimits_{#1}}%
          {\mathop{\vrule width 5pt height 3 pt depth -2.6pt
                  \kern -6pt \intop}\nolimits_{#1}}%
          {\mathop{\vrule width 5pt height 3 pt depth -2.6pt
                   \kern -6pt \intop}\nolimits_{#1}}}

\begin{document}

\title{Weak Quasicircles Have Lipschitz Dimension 1}

\author{David M. Freeman}
\address{University of Cincinnati Blue Ash College, Department of Math, Physics, and Computer Science, 9555 Plainfield Rd., Cincinnati, OH 45236, United States}
\email{david.freeman@uc.edu}


\date{\today}

\commby{Nageswari Shanmugalingam}

\begin{abstract}
We prove that the Lipschitz dimension of any bounded turning Jordan circle or arc is equal to 1. Equivalently, the Lipschitz dimension of any weak quasicircle or arc is equal to 1.
\end{abstract}

\maketitle

\section{Introduction}                                                  

In \cite{CK13}, Cheeger and Kleiner introduced the concept of Lipschitz dimension and proved deep results about metric spaces of Lipschitz dimension at most 1. Subsequently, in \cite{David19}, David further developed various  properties of Lipschitz dimension. While studying the non-invariance of Lipschitz dimension under quasisymmetric mappings in a general metric space setting, David asks in Question 8.7 of \cite{David19} if every quasisymmetric image of the unit interval (that is, a \textit{quasiarc}) has Lipschitz dimension equal to 1. We answer this question in the affirmative, thus demonstrating that the Lipschitz dimension of the unit interval is invariant under quasisymmetric homeomorphisms. In fact, we prove something stronger: the Lipschitz dimension of the unit interval is invariant under weakly quasisymmetric homeomorphisms. This can be derived from the following theorem, which constitutes the main result of this paper (see \rf{S:defs} for definitions). 

\begin{theorem}\label{T:main}
Bounded turning Jordan circles/arcs have Lipschitz dimension $1$.
\end{theorem}

Since bounded turning Jordan arcs need not be metrically doubling, not all bounded turning Jordan arcs are quasiarcs. On the other hand, every quasiarc is bounded turning. Analogous statements hold for Jordan circles. Therefore, in answer to \cite[Question 8.7]{David19}, we obtain the following corollary.

\begin{corollary}\label{C:main}
Quasicircles/arcs have Lipschitz dimension $1$.
\end{corollary}

We also point out results of Cheeger and Kleiner pertaining to spaces of Lipschitz dimension at most 1. In particular, via \cite[Theorem 1.7]{CK13}, we have the following corollary to \rf{T:main}.

\begin{corollary}\label{C:embed}
If $\Gamma$ is a bounded turning Jordan circle or arc, then there exists a measure space $(X,\mu)$ such that $\Gamma$ admits a bi-Lipschitz embedding into $L_1(X,\mu)$. 
\end{corollary}

Furthermore, via \cite[Theorem 1.11]{CK13}, we also have the following. 

\begin{corollary}\label{C:inverse}
If $\Gamma$ is a bounded turning Jordan circle or arc, then $\Gamma$ is bi-Lipschitz homeomorphic to an inverse limit of an admissible inverse system of graphs. 
\end{corollary}

Indeed, given a metric space $X$ and a Lipschitz light map $F:X\to\mathsf{R}$,  Cheeger and Kleiner explicitly construct an admissible inverse system of graphs whose inverse limit is bi-Lipschitz homeomorphic to $X$ (see \cite[Section 4]{CK13}).

\rf{T:main} is also relevant to the following result of David.

\begin{theorem}[Theorem 5.9 of \cite{David19}]\label{T:Carnot}
Let $\mathbb{G}$ denote a non-abelian Carnot group, and let $K\subset \mathbb{G}$ denote a compact subset of positive measure. Then the Lipschitz dimension of $K$ is equal to $\infty$. 
\end{theorem}

Since the Lipschitz dimension of a compact space is bounded from below by its topological dimension (see \cite[Observation 1.4]{David19}), given an integer $n\geq1$, we note that the Lipschitz dimension of the product of $n$ bounded turning Jordan arcs is at least $n$. Furthermore, the Lipschitz dimension of the product of $n$ bounded turning Jordan arcs is bounded from above by $n$ (here we use \rf{T:main} and \cite[Proposition 3.1]{David19}). Therefore, the Lipschitz dimension of the product of $n$ bounded turning Jordan arcs is equal to $n$. Since Lipschitz dimension is invariant under bi-Lipschitz homeomorphisms, we arrive at the following corollary.

\begin{corollary}\label{C:nonembed}
Let $\mathbb{G}$ denote a non-abelian Carnot group. If $K\subset \mathbb{G}$ is a compact subset of positive measure, then $K$ does not admit a bi-Lipschitz embedding into a product of finitely many bounded turning Jordan arcs.
\end{corollary}

In connection with the title of this paper, we also point out the following theorem of Meyer. In particular, all the above results pertaining to bounded turning Jordan circles/arcs can be understood as results about weak quasicircles/arcs. 

\begin{theorem}[Theorem 1.1 of \cite{Meyer11}]\label{T:Meyer}
A Jordan circle/arc $\Gamma$ is a weak quasicircle/arc if and only if $\Gamma$ is bounded turning.
\end{theorem}

Finally, we highlight the following result of Herron and Meyer, as it is foundational to our proof of \rf{T:main}.

\begin{theorem}[\cite{HM12}]
If $\Gamma$ is a bounded turning Jordan circle, then $\Gamma$ is bi-Lipschitz homeomorphic to some Jordan circle in $\mathcal{S}_1$. 
\end{theorem}

Here $\mathcal{S}_1$ is the collection of all Jordan circles given by dyadic diameter functions constructed using the snowflake parameter $\sigma=1$ (see \cite{HM12} for definitions). This result allows us to distort any given bounded turning Jordan circle into a form more amenable to the construction of a Lipschitz light map into $\mathsf{R}$.

The organization of this paper is as follows. In \rf{S:defs} we provide a few key definitions. Then, in \rf{S:prelims}, we investigate various aspects of Jordan circles in $\mathcal{S}_1$. In \rf{S:lip}, we construct a $1$-Lipschitz mapping from any Jordan circle $\Gamma\in\mathcal{S}_1$ into the unit circle. Finally, in \rf{S:light} we prove that this mapping is Lipschitz light via a series of technical lemmas.

\begin{acknowledge}
The author thanks Guy C. David for enlightening dialogue regarding the results and theoretical context of both this paper and \cite{David19}.
\end{acknowledge}

\section{Basic Definitions}\label{S:defs}

We write $\mathsf{N}$ to denote the set $\{0,1,2,\dots\}$ consisting of non-negative integers, and $\mathsf{R}$ to denote the Euclidean line.

Given two metric spaces $X$ and $Y$ and a constant $L\geq 1$, we say that an embedding $f:X\to Y$ is \textit{$L$-Lipschitz} provided that, for all points $x,y\in X$, we have $d(f(x),f(y))\leq L\,d(x,y)$. Furthermore, an embedding is \textit{$L$-bi-Lipschitz} if it is also true that $d(x,y)\leq L\,d(f(x),f(y))$.

An embedding $f:X\to Y$ is \textit{quasisymmetric} provided there exists a homeomorphism $\eta:[0,\infty)\to[0,\infty)$ such that, for all points $x,y,z\in X$ and $t\in[0,\infty)$,
\[d(x,y)\leq t\,d(x,z)\quad\text{implies}\quad d(f(x),f(y))\leq \eta(t)d(f(x),f(z)).\]
An embedding $f:X\to Y$ is \textit{weakly quasisymmetric} provided there exists a constant $H\geq 1$ such that, for all $x,y,z\in X$,
\[d(x,y)\leq d(x,z)\quad \text{implies}\quad d(f(x),f(y))\leq H d(f(x),f(z)).\]
While all quasisymmetries are weak quasisymmetries (with $H:=\eta(1)$), in general, a weak quasisymmetry need not be a quasisymmetry. We refer to the weak quasisymmetric image of the unit circle as a \textit{weak quasicircle}, and such an image of the closed unit interval as a \textit{weak quasiarc}. Thus every quasicircle/arc is a weak quasicircle/arc. Conversely, by \cite[Theorem 4.9]{TV80}, every weak quasicircle/arc that is \textit{metrically doubling} is a quasicircle/arc. Here we say that a space $X$ is metrically doubling if there exists $D\geq 1$ such that that any open ball of radius $2r>0$ can be covered by $D$ open metric balls of radius $r$. For additional information about weak quasicircles and quasiarcs, we refer the reader to \cite{Meyer11} and references therein.

A \textit{Jordan circle} $\Gamma$ is a homeomorphic image of the unit circle. Given two points $x,y\in \Gamma$, we write $\Gamma(x,y)$ to denote a component of $\Gamma\setminus\{x,y\}$ of minimal diameter. We write $\Gamma[x,y]$ to denote the topological closure of $\Gamma(x,y)$; thus $\Gamma[x,y]=\Gamma(x,y)\cup\{x,y\}$. Analogously, a \textit{Jordan arc} $\Gamma$ is a homeomorphic image of the closed unit interval. In this setting, given two points $x,y\in\Gamma$, we write $\Gamma(x,y)$ to denote the connected component of $\Gamma\setminus\{x,y\}$. Again, $\Gamma[x,y]=\Gamma(x,y)\cup\{x,y\}$. 

A Jordan circle or arc is said to be  \textit{bounded turning} provided that there exists a constant $C\geq 1$ such that, for all pairs of points $x,y\in\Gamma$, we have $\diam(\Gamma[x,y])\leq C\,d(x,y)$. In this case we say that $\Gamma$ is $C$-bounded turning. This property is at times referred to in the literature as \textit{linear connectivity}. 

Given a metric space $(X,d)$ and $\delta>0$, a \textit{$\delta$-sequence} in $X$ is a finite sequence of points $\{x_i\}_{i=0}^n$ such that, for each $0\leq i\leq n-1$, we have $d(x_i,x_{i+1})\leq \delta$. A subset $U\subset X$ is said to be $\delta$-connected if every pair of points in $U$ is contained in a $\delta$-sequence in $U$. A \textit{$\delta$-component} of $X$ is a maximal $\delta$-connected subset of $X$.

We say that a map $F:X\to Y$ is \textit{Lipschitz light} provided there exists $C\geq 1$ such that $F$ is $C$-Lipschitz, and, for every $r>0$ and every subset $E\subset Y$ with $\diam(E)\leq r$, the $r$-components of $F^{-1}(E)$ have diameter bounded above by $C\,r$. Here we employ the definition of Lipschitz light used in \cite[Definition 1.2]{David19}. As shown by David in \cite[Section 1.4]{David19}, this definitions is equivalent to \cite[Definition 1.14]{CK13} for maps into $\mathsf{R}^n$ (for $n\geq1$).

A metric space $X$ has \textit{Lipschitz dimension} $n\in\mathsf{N}$ if $n$ is minimal such that there exists a Lipschitz light map $F:X\to\mathsf{R}^n$.

\section{Preliminary Results}\label{S:prelims}

Following \cite{HM12}, we view the unit circle $\mathsf{S}$ as $[0,1]/\{0,1\}$, the closed unit interval whose endpoints are identified. We equip $\mathsf{S}$ with the arc-length metric $\lambda$. That is, for two points $s,t\in\mathsf{S}$ such that $0\leq s\leq t\leq 1$, we have
\[\lambda(s,t):=\min\{t-s,1-(t-s)\}.\]
The space $\mathsf{S}$ inherits the usual left-to-right orientation on $[0,1]$. 

Given $n\in\mathsf{N}$, we write $\mathcal{I}_n$ to denote the collection of $2^n$ closed dyadic intervals in $[0,1]$, each of length $2^{-n}$. We write $\hI_n$ to denote the collection $\cup_{m=0}^n\mathcal{I}_m$. Furthermore, we write $\hat{\mathcal{I}}$ to denote the collection $\cup_{n=0}^\infty\mathcal{I}_n$. Given an interval $I\in\hat{\mathcal{I}}$, we write $l(I)$ to denote the unique index $n\in\mathsf{N}$ such that $I\in \mathcal{I}_n$. For convenience, we use the language of a dyadic tree to describe intervals in $\hat{\mathcal{I}}$. In particular, given any $I\in\hat{\mathcal{I}}$, there are exactly two dyadic \textit{children} contained in $I$, and $I$ is contained in its unique dyadic \textit{parent} interval.

Similarly, we write $\mathcal{D}_n$ to denote the collection of $2^{n}$ dyadic endpoints of intervals in $\mathcal{I}_n$. For example, $\mathcal{D}_0=\{0\}=\{1\}$, $\mathcal{D}_1=\{0,1/2\}=\{1,1/2\}$, etc. Note that, for each $n\in\mathsf{N}$, we have $\mathcal{D}_n\subset\mathcal{D}_{n+1}$. We write $\mathcal{D}$ to denote $\cup_{n=0}^\infty\mathcal{D}_n$.

In order to utilize the catalogue $\mathcal{S}_1$, we rely upon notation and terminology from \cite{HM12}. Given a dyadic diameter function $\Delta$, the distance $d_\Delta$ on $\mathsf{S}$ is defined as 
\[d_\Delta(x,y):=\inf\sum_{k=1}^N\Delta(J_k),\]
where the infimum is taken over all chains $J_1,\dots, J_N$ of intervals from $\hat{\mathcal{I}}$ joining $x$ to $y$. That is, $\{x,y\}$ is contained in the connected set  $J_1\cup\dots\cup J_N$. For each $n\in\mathsf{N}$, we define a distance $d_n$ on $\mathsf{S}$ using the \textit{truncated} diameter function $\Delta_n$, defined as follows. For $m\leq n$ and $I\in\mathcal{I}_m$, we define $\Delta_n(I):=\Delta(I)$. For every $m> n$ and $I\in \mathcal{I}_m$, we inductively define $\Delta_n(I)=\frac{1}{2}\Delta_n(\tilde{I})$, where $\tilde{I}\in\mathcal{I}_{m-1}$ denotes the dyadic parent of $I$. In analogy with $d_\Delta$, we then define
\[d_n(x,y):=\inf\sum_{k=1}^N\Delta_n(J_k),\]
where the infimum is taken over all chains $\{J_k\}_{k=1}^N$ joining $x$ to $y$. We write $\Gamma_n$ to denote the metric space $(\mathsf{S},d_n)$, and $\Gamma$ to denote $(\mathsf{S},d_\Delta)$. For $n\in\mathsf{N}$, we write $\diam_n(E)$ to denote the $d_n$-diameter of a set $E\subset\mathsf{S}$. Furthermore, we write $\diam_\Delta(E)$ to denote the $d_\Delta$-diameter of $E$.


We say that a chain of dyadic intervals $\{I_i\}_{i=1}^N$ is \textit{minimal} provided that it consists of intervals with pairwise disjoint interiors and that no union of at least two distinct intervals from the chain forms an interval in $\hat{\mathcal{I}}$. In particular, if the union of intervals $\cup_{k=1}^MI_{i_k}$ from a minimal chain $\{I_i\}_{i=1}^N$ is equal to some interval $J\in\hat{\mathcal{I}}$, then $M=1$ and $J=I_{i_1}\in\{I_i\}_{i=1}^N$.

\begin{lemma}\label{L:minimal}
The definitions of $d_n$ and $d_\Delta$ are unchanged by the assumption that the chains of dyadic intervals utilized in these definitions are minimal.
\end{lemma}

\begin{proof}
We focus on the definition of $d_n$. The proof for $d_\Delta$ is the same. Suppose that $\{I_i\}_{i=1}^N$ is a chain of dyadic intervals joining $x$ and $y$ in $\mathsf{S}$. Suppose $I_j$ and $I_k$ have non-disjoint interiors. Since both $I_j$ and $I_k$ are dyadic, one must be a subset of the other. Without loss of generality, $I_j\subset I_k$. Therefore, the sum $\sum_{i=1}^N\Delta_n(I_i)$ can be decreased by eliminating the interval $I_j$ from $\{I_i\}_{i=1}^N$. It follows that $d_n$ can be defined only using chains consisting of intervals with pairwise disjoint interiors. 

Next, suppose there exists $M\geq2$ and a subcollection $\{I_{i_k}\}_{k=1}^M\subset \{I_i\}_{i=1}^N$ such that $J=\cup_{k=1}^MI_{i_k}\in\hat{\mathcal{I}}$. Since $\Delta_n(J)\leq\sum_{k=1}^M\Delta_n(I_{i_k})$, the sum $\sum_{i=1}^N\Delta_n(I_i)$ can be decreased by replacing the intervals $\{I_{i_k}\}_{k=1}^M$ in $\{I_i\}_{i=1}^N$ with the single interval $J$. Since $N<+\infty$, such a replacement can happen at most finitely many times. It follows that $d_n$ can be defined only using minimal chains.
\end{proof}

For use below, we record the following technical lemma.

\begin{lemma}\label{L:minimalchain}
Assume $\{I_i\}_{i=1}^N$ is a minimal chain of dyadic intervals indexed such that, for $1\leq i\leq N-1$, the right endpoint of $I_i$ is the left endpoint of $I_{i+1}$. Under this assumption, there is either a unique interval or a unique pair of adjacent intervals from $\{I_i\}_{i=1}^N$ of maximal $d_0$-diameter. Write $i_*$ to denote the index of such a maximal interval. If $i_*>1$, then  $l(I_i)$ is strictly decreasing for $i=1,\dots,i_*-1$. If $i_*<N$, then $l(I_i)$ is strictly increasing for $i=i_*+1,\dots,N$.
\end{lemma}

\begin{proof}
We may assume that $\sigma:=\bigcup_{i=1}^NI_i\not=\mathsf{S}$, else $N=1$ and $I_1=\mathsf{S}$. 

Suppose there are two distinct intervals $I_j$ and $I_k$ in $\{I_i\}_{i=1}^N$ of maximal $d_0$-diameter, where $j<k$ and $n:=l(I_j)=l(I_k)$. If these intervals are not adjacent, then the chain $\{I_{i}\}_{i=j+1}^{k-1}$ consisting of intervals from $\{I_i\}_{i=1}^N$ joins the right endpoint of $I_j$ to the left endpoint of $I_k$. Since $I_j$ and $I_k$ are not adjacent, the union $\cup_{i=j}^k I_i\subset \sigma$ contains at least three consecutive intervals from $\mathcal{I}_n$. Such a union must contain some interval $J$ from $\mathcal{I}_{n-1}$. It follows from the assumption that $\{I_i\}_{i=1}^N$ is minimal that the interval $J$ must be an element of $\{I_i\}_{i=1}^N$. However, this violates the assumption that $I_j$ and $I_k$ are of maximal $d_0$-diameter in $\mathcal{I}_n$. Therefore, the intervals $I_j$ and $I_k$ must be adjacent. 

To verify the second part of the lemma, suppose $i_*>1$. If $i_*=2$ then the desired conclusion is trivial, so we may assume that $i_*\geq3$. Let $1\leq i\leq i_*-2$, and write $m:=l(I_i)$. Since $I_i$ lies to the left of $I_{i_*}$ in $\sigma$ and $\diam_0(I_{i_*})$ is strictly larger than $\diam_0(I_i)$, the interval $J$ of $\mathcal{I}_m$ immediately to the right of $I_i$ is contained in $\sigma$. By the minimality of $\{I_i\}_{i=1}^N$, no union of at least two intervals from $\{I_i\}_{i=1}^N$ is equal to $J$. Therefore, either $J=I_{i+1}$ or $J$ is strictly contained in $I_{i+1}$. If $J=I_{i+1}$, then (lest we violate the minimality of $\{I_i\}_{i=1}^N$), the midpoint of $I_i\cup I_{i+1}$ is an element of $\mathcal{D}_{m-1}$, and so $I_i\cup I_{i+1}\not\in\mathcal{I}_m$. In this case, since $I_{i+1}$ lies to the left of $I_{i_*}$ and $\diam_0(I_{i_*})>\diam_0(I_{i+1})$, the interior of the interval $K$ in $\mathcal{I}_m$ immediately to the right of $I_{i+1}$ is disjoint from the interior of $I_{i_*}$. It follows from the minimality of $\{I_i\}_{i=1}^N$ that $K$ cannot be strictly contained in any element of $\{I_i\}_{i=1}^N$ (since any dyadic interval strictly containing $K$ must also contain $I_{i+1}$), and so (again using minimality), we conclude that $K=I_{i+2}$. But this nevertheless violates the minimality of $\{I_i\}_{i=1}^N$, since $I_{i+1}\cup I_{i+2}\in\mathcal{I}_{m-1}$. Therefore, $J$ is strictly contained in $I_{i+1}$, and so $l(I_{i+1})<l(I_i)$.

An analogous argument verifies the final assertion of the lemma.
\end{proof}

Since, for any $I\in\mathcal{I}$, we have $\Delta_n(I)\leq\Delta(I)$, it follows that, for any $x,y\in\mathsf{S}$,
\begin{equation}\label{E:leqn}
d_n(x,y)\leq d_\Delta(x,y).
\end{equation}
Therefore, for any set $E\subset \mathsf{S}$, we have $\diam_n(E)\leq \diam_\Delta(E)$. Furthermore, for any $n\in\mathsf{N}$ and $I\in\hI_n$ with endpoints $a,b$, via \cite[Lemma 3.1]{HM12}, we have
\begin{equation}\label{E:intn}
d_n(a,b)=\diam_{n}(I)=\Delta_n(I)=\Delta(I)=\diam_\Delta(I)=d_\Delta(a,b).
\end{equation}

\begin{lemma}\label{L:distconverge}
If $x,y\in \mathsf{S}$ and $n\in\mathsf{N}$, then 
\[d_\Delta(x,y)\leq d_n(x,y)+2\max\{\Delta(I)\,|\,I\in\mathcal{I}_n\}.\]
In particular, $d_n(x,y)\to d_\Delta(x,y)$.
\end{lemma}
\begin{proof}

Let $M(n):=\max\{\Delta(I)\,|\,I\in\mathcal{I}_n\}$. Fix $x,y\in \mathsf{S}$, $n\in\mathsf{N}$, and let $0<\varepsilon<M(n)$ be given. Let $\{I_i\}_{i=1}^N$ be a minimal chain of dyadic intervals joining $x$ and $y$, indexed as in the assumptions of \rf{L:minimalchain}, such that $\sum_{i=1}^N\Delta_n(I_i)<d_n(x,y)+\varepsilon$. If $\{I_i\}_{i=1}^N\subset\hI_n$, then we are done, because $\Delta=\Delta_n$ on $\hI_n$. If not, then let $I_{i_*}$ denote an interval from $\{I_i\}_{i=1}^N$ of maximal $d_0$-diameter, and write $m:=l(I_{i_*})$. If $x$ and $y$ are contained in adjacent intervals $J$ and $K$ from $\mathcal{I}_{n}$, then 
\[d_\Delta(x,y)\leq \Delta(J)+\Delta(K)\leq 2M(n)\leq d_n(x,y)+2M(n).\]
Therefore, we can assume that $x$ and $y$ are contained in non-adjacent intervals from $\mathcal{I}_n$. It follows from the minimality of $\{I_i\}_{i=1}^N$ that $m\leq n$. Therefore, either $l(I_1)\leq n$, or, by \rf{L:minimalchain}, there exists a maximal index $i$ such that $1\leq i_1<i_*$ and, if $1\leq i\leq i_1$, then $l(I_i)>n$. Similarly, either $l(I_N)\leq n$, or there exists a minimal index $i$ such that $i_*< i_2\leq N$ and, if $i_2\leq i\leq N$, then $l(I_i)>n$. Assume the existence of such $i_1$ and $i_2$ (else the following argument simplifies). Via \rf{L:minimalchain}, one can verify that the interval $\sigma_1:=\bigcup_{i=1}^{i_1}I_i$ is contained in some interval $J_1\in\mathcal{I}_n$ adjacent (on the left) to $I_{i_1+1}$. Similarly, $\sigma_2:=\bigcup_{i=i_2}^{N}I_i$ is contained in some interval $J_2\in\mathcal{I}_n$ adjacent (on the right) to $I_{i_2-1}$. Thus we have 
\begin{align*}
d_\Delta(x,y)&\leq \Delta(J_1)+\sum_{i=i_1+1}^{i_2-1}\Delta(I_i)+\Delta(J_2)=\Delta(J_1)+\sum_{i=i_1+1}^{i_2-1}\Delta_n(I_i)+\Delta(J_2)\\
&\leq\sum_{i=1}^N\Delta_n(I_i)+2M(n)<d_n(x,y)+\varepsilon+2M(n).
\end{align*}
Since $\varepsilon>0$ was arbitrary, we are done. 
\end{proof}

\section{Constructing a $1$-Lipschitz map $F_0:\Gamma\to\Gamma_0$}\label{S:lip}

Let $\Gamma$ denote a bounded turning Jordan circle or arc. Our first step towards the construction of a Lipschitz light map $F:\Gamma\to\mathsf{R}$ is to realize that it is sufficient to find a Lipschitz light map $F:\Gamma\to\Gamma_0$. This is because $\Gamma_0$ is easily seen to admit a Lipschitz light map into $\mathsf{R}$, and one can verify that the composition of a Lipschitz light map from $\Gamma$ to $\Gamma_0$ with a Lipschitz light map from $\Gamma_0$ into $\mathsf{R}$ is Lipschitz light (as noted in \cite[Section 5]{David19}). See also our comments at the outset of \rf{S:light}.

Next, we again recall the following result of Herron and Meyer.

\begin{theorem}[\cite{HM12}]
If $\Gamma$ is a bounded turning Jordan circle (or arc), then $\Gamma$ is bi-Lipschitz homeomorphic to some Jordan circle in $\mathcal{S}_1$ (or $\mathcal{S}_1'$).
\end{theorem}

Here $\mathcal{S}_1$ is defined as in \cite{HM12}. The collection $\mathcal{S}_1'$ can be analogously defined using dyadic diameter functions on the unit interval $[0,1]$. We remark that the validity of this extension of the main result of \cite{HM12} to Jordan arcs is pointed out by Herron and Meyer on page 605 of \cite{HM12}.

Since Lipschitz dimension is invariant under bi-Lipschitz homeomorphisms, we may work exclusively with Jordan circles in $\mathcal{S}_1$ (or arcs in $\mathcal{S}_1'$). We will only present the details for weak quasicircles; the details for weak quasiarcs are analogous. Thus, given a curve $\Gamma=(\mathsf{S},d_\Delta)\in\mathcal{S}_1$, we construct a Lipschitz light map $F:\Gamma\to\Gamma_0$. 

We will need the following map $f$ in order to achieve this goal, which we will refer to as a \textit{folding map}. Given an interval $I\subset \mathsf{S}$, divide $I$ into two dyadic subintervals of equal length with disjoint interiors, and denote these two subintervals by $I^0$ and $I^1$, respectively. We also divide $I$ into four consecutive dyadic subintervals of equal length with disjoint interiors, and denote these four subintervals by $I^{00}$, $I^{01}$, $I^{10}$, and $I^{11}$, respectively. Thus $I^0=I^{00}\cup I^{01}$ and $I^1=I^{10}\cup I^{11}$. Assume these intervals are indexed (in binary) such that adjacent intervals proceed consecutively along the positive orientation in $\mathsf{S}$. Finally, divide each of $I^{01}$ and $I^{10}$ into two dyadic subintervals of equal length with disjoint interiors, and denote these subintervals by $I^{010}$, $I^{011}$, $I^{100}$, and $I^{101}$, respectively. Thus $I^{01}=I^{010}\cup I^{011}$ and $I^{10}=I^{100}\cup I^{101}$. Again we index these intervals such that their order reflects the positive orientation of $\mathsf{S}$. The map $f:I\to I$ is defined by its action on these subintervals. It maps
\begin{itemize}
	\item{$I^{00}$ linearly onto $I^0$ in an orientation preserving manner,}
	\item{$I^{010}$ linearly onto $I^{10}$ in an orientation preserving manner,}
	\item{$I^{011}$ linearly onto $I^{10}$ in an orientation reversing manner,}
	\item{$I^{100}$ linearly onto $I^{01}$ in an orientation reversing manner,}
	\item{$I^{101}$ linearly onto $I^{01}$ in an orientation preserving manner, and}
	\item{$I^{11}$ linearly onto $I^1$ in an orientation preserving manner.}
\end{itemize}

We note that this definition of a folding map can be scaled linearly and applied to any interval. Thus, for any $n\in\mathsf{N}$, let $I\in\mathcal{I}_n$. If the two dyadic children $I'$ and $I''$ of $I$ satisfy $\Delta(I')=\Delta(I'')=\frac{1}{2}\Delta(I)$, then we define the map $f_n:I\to I$ to be the identity map. Thus $f_n$ is an isometry from $(I,d_{n+1})\to(I,d_n)$. If $\Delta(I')=\Delta(I'')=\Delta(I)$, then we define the map $f_n:I\to I$ to be a folding map. The map $f_n:\Gamma_{n+1}\to\Gamma_n$ is defined in this manner on each interval $I\in\mathcal{I}_n$. 

\begin{lemma}\label{L:fold}
For each $n\in\mathsf{N}$, the map $f_n:\Gamma_{n+1}\to\Gamma_n$ is $1$-Lipschitz.
\end{lemma}

\begin{proof}
We examine the image of an interval $I\in\hat{\mathcal{I}}$ under the map $f_n$. If $I\in \hI_n$, then $f_n(I)=I\in\hat{\mathcal{I}}$. If $I\in\hat{\mathcal{I}}\setminus\hI_{n+1}$, then $f_n(I)\in\hat{\mathcal{I}}$ and $\diam_n(f_n(I))\leq \diam_{n+1}(I)$. If $I\in\mathcal{I}_{n+1}$ (the only remaining possibility), then $f_n$ fixes the endpoints of $I$, and $I\subset f_n(I)$. In particular, if $f_n$ is the identity on $\tilde{I}$ (the dyadic parent of $I$), then $f_n(I)=I$. If $f_n$ is a folding map on $\tilde{I}$, then $f_n(I)$ is the union of $I$ with an interval $J\in\mathcal{I}_{n+2}$ that is adjacent to $I$. Moreover, it is straightforward to verify that 
\begin{align*}
\diam_n(f_n(I))&=\diam_n(I\cup J)=\diam_n(I)+\diam_n(J)\\
	&=\frac{3}{4}\diam_n(\tilde{I})=\frac{3}{4}\diam_{n+1}(I)<\diam_{n+1}(I).
\end{align*}
Therefore, given a chain $\{I_i\}_{i=1}^N$ of dyadic intervals, $\{f_n(I_i)\}_{i=1}^N$ can be written as a chain of dyadic intervals $\{I_j'\}_{j=1}^{N'}$ such that $\cup_{i=1}^Nf_n(I_i)=\cup_{j=1}^{N'}I_j'$, and
\[\sum_{i=1}^N\diam_n(f_n(I_i))=\sum_{j=1}^{N'}\diam_n(I_j')=\sum_{j=1}^{N'}\Delta_n(I_j').\]
Thus, if $\{I_i\}_{i=1}^N$ joins $x$ and $y$, then $\{I_j'\}_{j=1}^{N'}$ joins $f_n(x)$ and $f_n(y)$, and
\[d_n(f_n(x),f_n(y))\leq\sum_{j=1}^{N'}\Delta_n(I_j')=\sum_{i=1}^N\diam_n(f_n(I_i))\leq\sum_{i=1}^N\Delta_{n+1}(I_i).\] 
It follows from the definition of $d_{n+1}(x,y)$ that $d_n(f_n(x),f_n(y))\leq d_{n+1}(x,y)$. 
\end{proof}

Given any $m\leq n\in\mathsf{N}$, we define $F_{m,n}:=f_m\circ f_{m+1}\circ \dots \circ f_n:\Gamma_{n+1}\to \Gamma_m$. If $m=n$, then we understand $F_{m,m}=F_{n,n}$ to denote $f_n$. As a composition of $1$-Lipschitz maps (\rf{L:fold}), each map $F_{m,n}:\Gamma_{n+1}\to\Gamma_m$ is $1$-Lipschitz. Furthermore, this sequence of maps induces a $1$-Lipschitz map $F_m:(\mathcal{D},d_\Delta)\to(\Gamma_m,d_m)$. To see this, let $x\in \mathcal{D}$, and let $k\in\mathsf{N}$ be the smallest integer such that $x\in \mathcal{D}_k$. For any $n\geq k$, the map $f_n$ fixes the set $\mathcal{D}_k$. Therefore, if $n\geq k$, then $f_n(x)=x$. If $n\geq k>m$, we then observe that $\lim_{n\to+\infty}F_{m,n}(x)=F_{m,k}(x)$. In this case, define $F_m(x):=F_{m,k}(x)$. If $n>m\geq k$, then define $F_m(x):=x$. 

To see that $F_m$ thus defined is $1$-Lipschitz on $\mathcal{D}$, let $x,y$ denote any two points in $\mathcal{D}$. Choose $k\in\mathsf{N}$ such that $x,y\in \mathcal{D}_k$ and  $k>m$. Via \rf{E:leqn} and the fact that, for all $n\geq m$ the map $F_{m,n}$ is $1$-Lipschitz, we have
\[d_m(F_m(x),F_m(y))=d_m(F_{m,k}(x),F_{m,k}(y))\leq d_{k+1}(x,y)\leq d_\Delta(x,y).\]
Since $\Gamma$ is complete, $\mathcal{D}$ is dense in $\Gamma$ (cf.\,\cite[Section 3.2]{HM12}), and $F_m$ is Lipschitz on $\mathcal{D}$, it is then straightforward to extend $F_m$ such that $F_m:\Gamma\to\Gamma_m$ is $1$-Lipschitz. 

We note that we may also view the maps $F_{m,n}$ as acting on $\Gamma$. Moreover, it follows from \rf{E:leqn} that $F_{m,n}:\Gamma\to\Gamma_m$ is $1$-Lipschitz. With this in mind, we prove the following lemma.

\begin{lemma}\label{L:uniform}
For each $m\in\mathsf{N}$, the maps $F_{m,n}:\Gamma\to\Gamma_m$  uniformly converge to $F_m:\Gamma\to\Gamma_m$ as $n\to+\infty$.
\end{lemma}

\begin{proof}
Fix $m\in\mathsf{N}$ and $\varepsilon>0$. Choose $M\in\mathsf{N}$ such that
\[n\geq M \quad\text{implies} \quad\max\{\Delta(I)\,|\,I\in\mathcal{I}_n\}<\varepsilon/2.\] For any $x\in\Gamma$, there exists a nested sequence of dyadic intervals $I_n\in\mathcal{I}_n$ such that, for every $n\in\mathsf{N}$, we have $I_n\subset I_{n-1}$ and $x\in I_n$. Furthermore, there exists $x_n\in\mathcal{D}_n$ such that $x_n\in I_n$, and so $d_\Delta(x_n,x)\to0$. Assuming $n>\max\{m,M\}$, we have $F_m(x_{n+j})=F_{m,n+j}(x_{n+j})$. If $j=0$, write $w_{n,0}:=x_n$. If $j\geq1$, write $w_{n,j}:=f_{n+1}\circ\dots \circ f_{n+j}(x_{n+j})$. In either case, we note that $F_m(x_{n+j})=F_{m,n+j}(x_{n+j})=F_{m,n}(w_{n,j})$. Since $f_{n+j}(I_n)=I_n$, we have $w_{n,j}\in I_n$. Thus, 
\[F_m(x)=\lim_{j\to+\infty} F_m(x_{n+j})=\lim_{j\to+\infty}F_{m,n}(w_{n,j})\in F_{m,n}(I_n).\]
Combining these observations, we find that, for $n\geq \max\{m,M\}$, we have
\begin{align*}
d_m(F_{m,n}(x),F_m(x))&\leq d_m(F_{m,n}(x),F_{m,n}(x_{n}))+d_m(F_{m,n}(x_{n}),F_m(x))\\
&\leq d_\Delta(x,x_{n})+\diam_m(F_{m,n}(I_n))\\
&\leq 2\diam_\Delta(I_n)\leq 2\max\{\Delta(I)\,|\,I\in\mathcal{I}_n\}<\varepsilon.
\end{align*}
It follows that $F_{m,n}:\Gamma\to\Gamma_m$ is uniformly convergent to $F_m:\Gamma\to\Gamma_m$. 
\end{proof}

\section{Proving that $F_0:\Gamma\to\Gamma_0$ is Lipschitz light}\label{S:light}

Our goal in this section is to verify the existence of a constant $C\geq 1$ such that, for any subset $E\subset \Gamma_0$, the $\diam_0(E)$-components of $F_0^{-1}(E)$ have $d_\Delta$-diameter bounded above by $C\diam_0(E)$. Via the following lemma, this will be sufficient to prove that $F_0:\Gamma\to\Gamma_0$ is Lipschitz light, and thus (via the comments at the outset of \rf{S:lip}) that $\Gamma$ has Lipschitz dimension equal to $1$. 

\begin{lemma}\label{L:def}
Suppose there exists a constant $C\geq 1$ such that $F:\Gamma\to\Gamma_0$  is $C$-Lipschitz, and, for any subset $E\subset \Gamma_0$ such that $\diam_0(E)>0$, the $\diam_0(E)$-components of $F^{-1}(E)$ have $d_\Delta$-diameter bounded by $C\,\diam_0(E)$. This implies that, for any $r>0$ and any subset $E\subset \Gamma_0$ satisfying $\diam_0(E)\leq r$, the $r$-components of $F^{-1}(E)$ have $d_\Delta$-diameter bounded by $C'r$, for $C':=\max\{C,8\}$. 
\end{lemma}

\begin{proof}
Let $r>0$, and let $E\subset \Gamma_0$ be such that $\diam_0(E)\leq r$. We may assume that $E$ is compact, and that $\diam_0(E)<r$. If $r\geq 1/8$, then we note that $\diam_0(F^{-1}(E))\leq \diam_\Delta(\Gamma)\leq 1\leq 8r$. Thus, we may assume that $r<1/8$. 

We claim that $E$ is contained in a subset $E'\subset \Gamma_0$ such that $\diam_0(E')=r$. To see this, we modify the argument employed in \cite[Remark 1.9]{David19}. We first note that $d_0=\lambda$ (the normalized length distance defined in \rf{S:prelims}). Therefore, given $x\in\Gamma_0$, the subset $I(x):=\{y\in \Gamma_0\,|\,d_0(x,y)\leq 1/8\}$ is isometric to some interval in $\mathsf{R}$ of length $1/4$. If $x\in E$, then $E\subset I(x)$. Let $a$ and $b$ denote the first and last points in $E$ along the interval $I(x)$. Thus $d_0(a,b)=\diam_0(E)<r$. Let $c\in I(x)$ be such that $d_0(a,c)=r<1/8$ and $\Gamma_0[a,b]\subset\Gamma_0[a,c]=:E'$. Then $E\subset E'$ and $\diam_0(E')=r$. Clearly, $r$-components of $F^{-1}(E)$ are contained in $r$-components of $F^{-1}(E')$, which (by assumption) have $d_\Delta$-diameter bounded by $Cr$. 
\end{proof}

With the above lemma in hand, we begin our proof that $F_0:\Gamma\to\Gamma_0$ is Lipschitz light. Fix $E\subset \Gamma_0$ such that $\diam_0(E)>0$, and let $M^*\in\mathsf{N}$ be maximal such that 
\begin{equation}\label{E:delta_size}
2^{-M^*-1}\leq\diam_0(E)<2^{-M^*}.
\end{equation}
We may assume that $M^*\geq 3$, else $\diam_\Delta(F_0^{-1}(E))\leq\diam_\Delta(\Gamma)\leq8\diam_0(E)$. By definition of $M^*$, there exist two adjacent dyadic subintervals $I,J\in\mathcal{I}_{M^*}$ such that $E\subset I\cup J$. In fact, $E$ may be contained in a single element of $\mathcal{I}_{M^*}$, but it will do no harm to assume $E$ is contained in the union of two such intervals. 

We claim that it is sufficient to examine pre-images of $H:=I\cup J$. Indeed, given any $\delta>0$, the $\delta$-components of $F^{-1}(E)$ are contained in $\delta$-components of $F^{-1}(H)$. 

For the remainder of this section, we set $\delta:=\diam_0(E)$.

\begin{lemma}\label{L:projection}
Given $n\in\mathsf{N}$ and $U\subset \Gamma_0$, we have $F_{n+1}(F_0^{-1}(U))=F_{0,n}^{-1}(U)$. 
\end{lemma}

\begin{proof}
Suppose $x\in F_{n+1}(F_0^{-1}(U))$, so $x=F_{n+1}(w)$ for some $w\in F_0^{-1}(U)$. Then 
\[F_{0,n}(x)=F_{0,n}(F_{n+1}(w))=\lim_{m\to\infty}F_{0,n}(F_{n+1,m}(w))=\lim_{m\to\infty}F_{0,m}(w)=F_0(w).\]
Since $F_0(w)\in U$, it follows that $F_{n+1}(F_0^{-1}(U))\subset F_{0,n}^{-1}(U)$. 

Next, let $x\in F_{0,n}^{-1}(U)$. Write $z_{n+1}:=x$, and choose a point $z_{n+2}\in f_{n+1}^{-1}(z_{n+1})$ so that $f_{n+1}(z_{n+2})=z_{n+1}$. Inductively, for each $k\geq2$, define $z_{n+k}$ such that 
\[f_{n+k-1}(z_{n+k})=z_{n+k-1}.\]
We claim that the sequence $\{z_{n+k}\}_{k=1}^\infty$ is Cauchy with respect to  $d_\Delta$, and thus convergent to some point $z\in \Gamma$. Indeed, for any $1\leq i< j$, we note that
\[z_{n+i}=f_{n+i}\circ\dots\circ f_{n+j-1}(z_{n+j}).\]
Let $I\in\mathcal{I}_{n+i}$ denote an interval containing $z_{n+j}$. For all $k\in\mathsf{N}$, we have $f_{n+i+k}(I)=I$. Therefore, $z_{n+i}\in I$, and so
\[d_\Delta(z_{n+i},z_{n+j})\leq\diam_\Delta(I)\leq\max\{\Delta(J)\,|\,J\in\mathcal{I}_{n+i}\}.\]
Since $\max\{\Delta(J)\,|\,J\in\mathcal{I}_{n+i}\}\to0$ as $i\to\infty$, our claim follows.

Next, we claim that $z=\lim_{k\to+\infty}z_{n+k}\in F_0^{-1}(U)$. Via \rf{L:uniform}, we have
\begin{align*}
F_0(z)&=\lim_{m\to+\infty}F_{0,n+m-1}(z_{n+m})\\
&=\lim_{m\to+\infty}F_{0,n}(F_{n+1,n+m-1}(z_{n+m}))=\lim_{m\to+\infty}F_{0,n}(z_{n+1})=F_{0,n}(x)\in U.
\end{align*}
Finally, we claim $F_{n+1}(z)=x$. Again via \rf{L:uniform}, we note that
\[F_{n+1}(z)=\lim_{m\to\infty}F_{n+1,m}(z_{m+1})=\lim_{m\to\infty}f_{n+1}\circ\dots\circ f_{m}(z_{m+1})=z_{n+1}=x\]
Therefore, $x\in F_{n+1}(F_0^{-1}(U))$, and so $F_{0,n}^{-1}(U)\subset F_{n+1}(F_0^{-1}(U))$.
\end{proof}

\begin{lemma}\label{L:component_proj}
Given $m\leq n\in\mathsf{N}$ and $x\in \Gamma$, we have $F_{m,n}(F_{n+1}(x))= F_m(x)$.
\end{lemma}

\begin{proof}
$F_{m,n}(F_{n+1}(w))=\lim_{k\to\infty}F_{m,n}(F_{n+1,k}(w))=\lim_{k\to\infty}F_{m,k}(w)$.
\end{proof}

Let $W$ denote any fixed $\delta$-component of $F_0^{-1}(H)$. By \rf{L:projection}, we have $F_{n+1}(W)\subset F_{0,n}^{-1}(H)$. Given any $n\in\mathsf{N}$, via \rf{L:uniform}, the set $F_{n+1}(W)$ is $\delta$-connected in $\Gamma_{n+1}$. In particular, it is contained in a single $\delta$-component of $F_{0,n}^{-1}(H)$ in $\Gamma_{n+1}$. We denote this $\delta$-component by $V_{n+1}$. Thus, for every $n\geq1$, write $V_n$ to denote the $\delta$-component of $F_{0,n-1}^{-1}(H)$ containing $F_{n}(W)$. We also write $V_0:=H$. 

\begin{lemma}\label{L:Vs}
For $n\in\mathsf{N}$, we have $f_n(V_{n+1})\subset V_n$. Furthermore, the set $V_{n+1}$ is a $\delta$-component of $f_n^{-1}(V_n)$. 
\end{lemma}

\begin{proof}
If $n=0$, then $V_1\subset F_0^{-1}(H)=f_0^{-1}(H)$, and so $f_0(V_1)\subset H=V_0$. We assume $n\geq1$. Via \rf{L:component_proj}, we have $F_n(W)= f_n(F_{n+1}(W))\subset f_n(V_{n+1})\subset F_{0,n-1}^{-1}(H)$. By definition, $F_n(W)\subset V_n\subset F_{0,n-1}^{-1}(H)$. Therefore, the sets $V_n$ and $f_n(V_{n+1})$ are both subsets of $F_{0,n-1}^{-1}(H)$ and have  non-trivial intersection. Since $f_n:\Gamma_{n+1}\to\Gamma_n$ is $1$-Lipschitz, the set $f_n(V_{n+1})$ is $\delta$-connected. Since $V_n$ is a maximal $\delta$-connected subset of $F_{0,n-1}^{-1}(H)$, we must have $f_n(V_{n+1})\subset V_n$.

Since $V_{n+1}\subset f_n^{-1}(V_n)$ and $V_{n+1}$ is $\delta$-connected,  $V_{n+1}$ is contained in a single $\delta$-component of $f_n^{-1}(V_n)\subset F_{0,n}^{-1}(H)$. Since $V_{n+1}$ is a maximal $\delta$-connected subset of $F_{0,n}^{-1}(H)$, the set $V_{n+1}$ is equal to a single $\delta$-component of $f_n^{-1}(V_n)$.  
\end{proof}

\begin{lemma}\label{L:goal}
There exists $N^*\in\mathsf{N}$ such that, if $n\geq N^*$, then 
\[\diam_\Delta(W)\leq \diam_n(V_n)+\delta.\]
\end{lemma}

\begin{proof}
Choose $N^*\in\mathsf{N}$ such that, for $n\geq N^*$, we have $\max\{\Delta(I)\,|\,I\in\mathcal{I}_n\}<\delta/4$. Let $x,y\in W$. Since, for $n\in\mathsf{N}$, the map $F_n$ fixes elements of $\mathcal{I}_n$, we note that 
\[d_\Delta(x,y)\leq d_\Delta(F_n(x),F_n(y))+2\max\{\Delta(I)\,|\,I\in\mathcal{I}_n\}<d_\Delta(F_n(x),F_n(y))+\delta/2.\]
Furthermore, via \rf{L:distconverge}, we also have (for $n\geq N^*$)
\[d_\Delta(F_n(x),F_n(y))\leq \diam_\Delta(F_n(W))\leq \diam_n(F_n(W))+\delta/2.\]
Since $F_n(W)\subset V_n$, we conclude that, for any $n\geq N^*$ and any $x,y\in W$, we have
\[d_\Delta(x,y)\leq \diam_n(V_n)+\delta.\]
It follows that $\diam_\Delta(W)\leq \diam_n(V_n)+\delta$.
\end{proof}

\begin{lemma}\label{L:big}
If, for some $n\in\mathsf{N}$, the set $V_n$ is contained in an interval $I_n\in\mathcal{I}_n$ and $\diam_n(V_n)\geq\frac{1}{4}\diam_n(I_n)$, then, for any $k\in\mathsf{N}$, we have $\diam_{n+k}(V_{n+k})\leq4\diam_n(V_n)$. 
\end{lemma}

\begin{proof}
We first note that $V_{n+1}\subset I_n$, since, by \rf{L:Vs}, $f_n(V_{n+1})\subset V_n\subset I_n$, and  $f_n(I_n)=I_n$. Via induction, for all $k\in\mathsf{N}$, we have $V_{n+k}\subset I_n$. Therefore, via \rf{E:intn}, we have $\diam_{n+k}(V_{n+k})\leq\diam_{n+k}(I_n)=\diam_n(I_n)\leq 4\diam_n(V_n)$.
\end{proof}

\begin{lemma}\label{L:cases}
Suppose that, for some $n\in\mathsf{N}$, we have
\begin{enumerate}
	\item{$\diam_n(V_n)=\diam_0(V_0)$}
	\item{$V_n$ is the union of two adjacent intervals from $\mathcal{I}_m$, for some $m\geq M^*$,}
	\item{$V_n$ is not symmetric about a point in $\mathcal{D}_n$,}
	\item{$V_n$ is contained in a single interval $I_n\in\mathcal{I}_n$, and}
	\item{$\diam_n(V_n)\leq\frac{1}{4}\diam_n(I_n)$.}
\end{enumerate}
Under these assumptions, $\diam_{n+1}(V_{n+1})\leq 2\diam_n(V_n)$. If $\diam_{n+1}(V_{n+1})>\diam_n(V_n)$, then $\diam_{n+1}(V_{n+1})=2\diam_n(V_n)$ and $V_{n+1}$ is symmetric about a point in $\mathcal{D}_{n+3}\setminus\mathcal{D}_{n+1}$. If, on the other hand, $\diam_{n+1}(V_{n+1})<\diam_n(V_n)$, then $\diam_{n+1}(V_{n+1})=0$ and $V_{n+1}$ is a point in $\mathcal{D}_{n+3}\setminus\mathcal{D}_{n+2}$. 
\end{lemma}

\begin{proof}
Note that Assumption (3) follows from Assumption (4) when $n\geq1$; we list Assumption (3) to address the case that $n=0$. 

We use binary superscripts to index the four second-generation dyadic sub-intervals $I_n^{00}$, $I_n^{01}$, $I_n^{10}$, and $I_n^{11}$ in $I_n$ such that they proceed consecutively along the positive orientation in $I_n$.

If $f_n$ is the identity on $I_n$, then the lemma is trivial. Therefore, we assume that $f_n$ is a folding map on $I_n$, and we consider the cases below. We preface this case analysis with the reminder that
\[\delta<\frac{1}{2^{M^*}}=\frac{1}{2}\diam_0(V_0)=\frac{1}{2}\diam_n(V_n)\leq\frac{1}{8}\diam_n(I_n).\]

Case 1: $V_n\subset I_n^{00}$. In this case, $f_n^{-1}(V_n)$ consists of either one $\delta$-component or (if $V_n$ contains the right endpoint of $I^{00}_n$) it consists of two. If one, then, via \rf{L:Vs}, we have $V_{n+1}=f_n^{-1}(V_n)\subset I_n^{00}$ and $\diam_{n+1}(V_{n+1})=\diam_n(V_n)$. If two, then one $\delta$-component is contained in $I^{00}_n$ and satisfies $\diam_{n+1}(V_{n+1})=\diam_n(V_n)$ while the other is a single point located at the midpoint of $I^{10}_n$.

Case 2: $V_n\subset I_n^{01}$. In this case, there are at most three $\delta$-components of $f_n^{-1}(V_n)$: one in $I_n^{00}$ and either one or two in $I_n^{10}$. The component in $I_n^{00}$ has $d_{n+1}$-diameter equal to $\diam_n(V_n)$. If there are two components in $I_n^{10}$, then they each have $d_{n+1}$-diameter equal to $\diam_n(V_n)$. If there is one component in $I_n^{10}$, then it has $d_{n+1}$-diameter equal to $2\diam_n(V_n)$, and it is symmetric about the midpoint of $I_n^{10}$. Via \rf{L:Vs}, if $\diam_{n+1}(V_{n+1})>\diam_n(V_n)$, then $V_{n+1}$ is symmetric about a point in $\mathcal{D}_{n+3}\setminus\mathcal{D}_{n+1}$ and $\diam_{n+1}(V_{n+1})=2\diam_n(V_n)$.

Case 3: $V_n\subset I_n^{10}$. By symmetry, we can apply an argument parallel to that used in Case 2 to conclude that $\diam_{n+1}(V_{n+1})\leq 2\diam_n(V_n)$. Furthermore, if $\diam_{n+1}(V_{n+1})>\diam_n(V_n)$, then $V_{n+1}$ is symmetric about a point in $\mathcal{D}_{n+3}\setminus\mathcal{D}_{n+1}$ and $\diam_{n+1}(V_{n+1})=2\diam_n(V_n)$.

Case 4. $V_n\subset I_n^{11}$. By symmetry, we can apply an argument parallel to that used in Case 1 to conclude that, either $V_{n+1}\subset I_n^{11}$ has diameter equal to $\diam_n(V_n)$, or $V_{n+1}$ is a single point at the midpoint of $I_n^{01}$.

Case 5: $V_n$ is symmetric about a point in $\mathcal{D}_{n+2}\setminus\mathcal{D}_{n+1}$. In this case, $f_n^{-1}(V_n)$ consists of two $\delta$-components. One is contained in $I_n^{00}$ (or $I_n^{11}$), and the other is contained in $I_n^{10}$ (or $I_n^{01}$). Each component has $d_{n+1}$-diameter equal to $\diam_n(V_n)$. Via \rf{L:Vs}, $\diam_{n+1}(V_{n+1})=\diam_n(V_n)$.

Case 6: $V_n$ is symmetric about a point in $\mathcal{D}_{n+1}\setminus\mathcal{D}_n$. In this case, there are three $\delta$-components of $f_n^{-1}(V_n)$, and each has $d_{n+1}$-diameter equal to $\diam_n(V_n)$. We note that one of these $\delta$-components is symmetric about a point in $\mathcal{D}_{n+1}\setminus\mathcal{D}_n$. In particular, this is the only case in which $V_{n+1}$ might not be contained in a single interval from $\mathcal{I}_{n+1}$. Via \rf{L:Vs}, $\diam_{n+1}(V_{n+1})=\diam_n(V_n)$.

Having exhausted the possible cases, we conclude the proof of the lemma.
\end{proof}

\begin{lemma}\label{L:none}
If there exists $K\in\mathsf{N}$ such that, for all $k\leq K$, the set $V_k$ is not symmetric about a point in $\mathcal{D}_{k+2}$, then, either there exists $n\geq N^*$ (for $N^*$ as in \rf{L:goal}) such that $\diam_n(V_n)\leq 16\delta$, or, for all $k\leq K$,
\begin{enumerate}
	\item[$(k.1)$]{$\diam_k(V_k)=\diam_0(V_0)$,}
	\item[$(k.2)$]{$V_k$ is the union of two adjacent intervals from $\mathcal{I}_{m}$ for some $m\geq M^*$,}
	\item[$(k.3)$]{$V_k$ is contained in a single interval $I_k\in\mathcal{I}_k$, and}
	\item[$(k.4)$]{$\diam_k(V_k)\leq \frac{1}{4}\diam_k(I_k)$.}
\end{enumerate}
\end{lemma}

\begin{proof}
Suppose $K\in\mathsf{N}$ is such that, for all $k\leq K$, no set $V_k$ is symmetric about a point in $\mathcal{D}_{k+2}$. In preparation for an inductive argument, we affirm the base case  $k=0\leq K$. Indeed, for $V_0=H$, we have
\begin{itemize}
	\item[(0.1)]{$\diam_0(V_0)=\diam_0(V_0)$,}
	\item[(0.2)]{$V_0$ is the union of two adjacent intervals from $\mathcal{I}_{M^*}$, and}
	\item[(0.3)]{$V_0$ is contained in a single interval $I_0\in\mathcal{I}_0$.}
	\item[(0.4)]{$\diam_0(V_0)\leq \frac{1}{4}\diam_0(I_0)$.}
\end{itemize}
Here we recall that $M^*\geq 3$. To proceed, we assume that, either there exists $n\geq N^*$ such that $\diam_n(V_n)\leq 16\delta$, or, for all $n\leq k-1\leq K-1$, we have
\begin{itemize}
	\item[$(n.1)$]{$\diam_n(V_n)=\diam_0(V_0)$,}
	\item[$(n.2)$]{$V_n$ is the union of two adjacent intervals from $\mathcal{I}_{m}$, for some $m\geq M^*$,}
	\item[$(n.3)$]{$V_n$ is contained in a single interval $I_n\in\mathcal{I}_n$, and}
	\item[$(n.4)$]{$\diam_n(V_n)\leq \frac{1}{4}\diam_n(I_n)$.}
\end{itemize}
Therefore, either there exists $n\geq N^*$ such that $\diam_n(V_n)\leq 16\delta$, or we satisfy the assumptions of \rf{L:cases} for $V_{k-1}$. Since $V_k$ is not symmetric about $\mathcal{D}_{k+2}$, \rf{L:cases} tells us that 
\begin{itemize}
	\item[$(k.1)$]{$\diam_k(V_k)=\diam_{k-1}(V_{k-1})$.}
\end{itemize}
We note that, if $f_{k-1}$ is the identity on $I_{k-1}$, then $V_k=V_{k-1}$. If $f_{k-1}$ is a folding map on $I_{k-1}$, then $V_k$ is the union of two adjacent intervals in $\mathcal{I}_{m+1}$ (here we are using (k.1)). In either case, 
\begin{itemize}
	\item[$(k.2)$]{$V_k$ is the union of two adjacent intervals in $\mathcal{I}_{m}$, for some $m\geq M^*$.}
\end{itemize}
Furthermore, since $V_{k-1}$ is not symmetric about a point in $\mathcal{D}_{k}$, Case 6 (in the proof of \rf{L:cases}) cannot occur. It follows that
\begin{itemize}
	\item[$(k.3)$]{$V_k$ is contained in a single interval $I_k\in\mathcal{I}_k$.}
\end{itemize}
Furthermore, if $\diam_k(V_k)>\frac{1}{4}\diam_k(I_k)$, then, (since $V_k$ is the union of two adjacent dyadic intervals) we must have either $\diam_k(V_k)=\frac{1}{2}\diam_k(I_k)$ or $\diam_k(V_k)=\diam_k(I_k)$. Since, by assumption, $V_k$ is not symmetric about a point in $\mathcal{D}_{k+2}$, neither case can occur. Therefore,
\begin{itemize}
	\item[$(k.4)$]{$\diam_k(V_k)\leq\frac{1}{4}\diam_k(I_k)$.}
\end{itemize}

Thus we conclude our inductive argument, and the proof of the lemma.
\end{proof}

\begin{lemma}\label{L:symm_next}
Suppose there exist $n\leq K\in\mathsf{N}$ such that, for all $n\leq k\leq K$, the set $V_{k}$ is not symmetric about a point in $\mathcal{D}_{k+1}\setminus\mathcal{D}_{k}$, and $n\in\mathsf{N}$ is such that
\begin{enumerate}
	\item[$(n.1)$]{Either $\diam_n(V_n)=0$ or $\diam_0(V_0)\leq\diam_n(V_n)\leq2\diam_0(V_0)$}
	\item[$(n.2)$]{$V_n$ is symmetric about a point in $\mathcal{D}_{n+2}\setminus\mathcal{D}_{n+1}$,}
	\item[$(n.3)$]{$V_n$ is contained in a single interval $I_n\in\mathcal{I}_n$, and}
	\item[$(n.4)$]{$\diam_n(V_n)\leq \frac{1}{4}\diam_n(I_n)$.}
\end{enumerate}
Under these assumptions on $n$ and $K$, for all $n\leq k\leq K$, it is true that
\begin{enumerate}
	\item[$(k.1)$]{$\diam_k(V_k)=\diam_n(V_n)$.}
	\item[$(k.2)$]{$V_{k}$ is symmetric about a point in $\mathcal{D}_{k+2}\setminus\mathcal{D}_{k+1}$,}
	\item[$(k.3)$]{$V_{k}$ is contained in a single interval $I_{k}\in\mathcal{I}_{k}$, and}
	\item[$(k.4)$]{$\diam_{k}(V_{k})\leq \frac{1}{4}\diam_n(I_{k})$.}
\end{enumerate}
\end{lemma}

\begin{proof} By way of induction, we first note that the base case $k=n$ is included in our assumptions. Thus, we assume that $K>n$ and, for all $n\leq j\leq k-1\leq K-1$,
\begin{itemize}
	\item[$(j.1)$]{$\diam_j(V_j)=\diam_n(V_n)$,}
	\item[$(j.2)$]{$V_{j}$ is symmetric about a point in $\mathcal{D}_{j+2}\setminus\mathcal{D}_{j+1}$,}
	\item[$(j.3)$]{$V_{j}$ is contained in a single interval $I_{j}\in\mathcal{I}_{j}$, and}
	\item[$(j.4)$]{$\diam_{j}(V_{j})\leq \frac{1}{4}\diam_{j}(I_{j})$.}
\end{itemize}
We prove that the analogous conclusions hold for $V_{k}$. Indeed, if $f_{k-1}$ is the identity on $I_{k-1}$, then $V_{k}=V_{k-1}$ is symmetric about a point in $\mathcal{D}_{k+1}\setminus\mathcal{D}_k$. Since, by assumption, this cannot occur, we only need to consider the case that $f_{k-1}$ is a folding map on $I_{k-1}$. Since $\delta<\frac{1}{2}\diam_0(V_0)\leq\frac{1}{8}\diam_{k-1}(I_{k-1})$, it is clear that
\begin{enumerate}
	\item[$(k.1)$]{$\diam_k(V_k)=\diam_{k-1}(V_{k-1})$,}
	\item[$(k.2)$]{$V_{k}$ is symmetric about a point in $\mathcal{D}_{k+2}\setminus\mathcal{D}_{k+1}$,}
	\item[$(k.3)$]{$V_{k}$ is contained in a single interval $I_{k}\in\mathcal{I}_{k}$, and}
	\item[$(k.4)$]{$\diam_{k}(V_{k})\leq \frac{1}{4}\diam_{k}(I_{k})$.}
\end{enumerate}
This completes the inductive argument, and the proof of the lemma.
\end{proof}

\begin{lemma}\label{L:not_middle}
Suppose there exist $n\leq K\in\mathsf{N}$ such that, for all $n\leq k\leq K$, the set $V_{k}$ is not symmetric about a point in $\mathcal{D}_{k}$, and $n$ is such that
\begin{enumerate}
	\item[$(n.1)$]{Either $\diam_n(V_n)=0$ or $\diam_0(V_0)\leq\diam_n(V_n)\leq2\diam_0(V_0)$,}
	\item[$(n.2)$]{$V_n$ is symmetric about a point in $\mathcal{D}_{n+1}\setminus\mathcal{D}_n$,}
	\item[$(n.3)$]{$V_n$ is contained in a single interval $I_n\in\mathcal{I}_n$, and}
	\item[$(n.4)$]{$\diam_n(V_n)\leq\frac{1}{4}\diam_n(I_n)$.}
\end{enumerate}
Under these assumptions, for all $n\leq k\leq K$, 
\begin{enumerate}
	\item[$(k.1)$]{$\diam_k(V_k)=\diam_n(V_n)$.}
	\item[$(k.2)$]{$V_k$ is symmetric about a point in $\mathcal{D}_{k+1}$,}
	\item[$(k.3)$]{$V_k$ is contained in a single interval $I_k\in\mathcal{I}_k$, and}
	\item[$(k.4)$]{$\diam_k(V_k)\leq\frac{1}{4}\diam_n(I_k)$.}
\end{enumerate}
\end{lemma}

\begin{proof}
The proof consists of a straightforward inductive argument similar to that used to prove \rf{L:symm_next}. For the sake of brevity, we omit the details.
\end{proof}

We are now ready to prove the following, which, via \rfs{L:goal} and \ref{L:def}, will be sufficient to prove that $F_0:\Gamma\to\Gamma_0$ is Lipschitz light.

\begin{lemma}\label{L:final}
There exists $n\in\mathsf{N}$ such that $n\geq N^*$ and $\diam_n(V_n)\leq 128\delta$.
\end{lemma}

\begin{proof}
If there is no index $n_1$ for which $V_{n_1}$ is symmetric about a point in $\mathcal{D}_{n_1+2}$, then, by \rf{L:none}, we conclude that $\diam_N(V_N)=\diam_0(V_0)\leq 4\delta$. Therefore, we may assume $n_1$ is the minimal such index. We first consider the case that $n_1\geq1$. By the definition of $n_1$ and \rf{L:none}, we may assume that
\begin{enumerate}
	\item{$\diam_{n_1-1}(V_{n_1-1})=\diam_0(V_0)$,}
	\item{$V_{n_1-1}$ is the union of two adjacent intervals from $\mathcal{I}_{m}$, for some $m\geq M^*$,}
	\item{$V_{n_1-1}$ is contained in a single interval $I_{n_1-1}\in\mathcal{I}_{n_1-1}$, and}
	\item{$\diam_{n_1-1}(V_{n_1-1})\leq\frac{1}{4}\diam_{n_1-1}(I_{n_1-1})$.}
\end{enumerate}
Via \rf{L:Vs}, it follows from the definition of $f_{n_1-1}$ and the minimality of $n_1$ that $V_{n_1}$ is symmetric about a point in $\mathcal{D}_{n_1+2}\setminus\mathcal{D}_{n_1+1}$, and, either $\diam_{n_1}(V_{n_1})=0$, or 
\[\diam_0(V_0)\leq\diam_{n_1}(V_{n_1})\leq2\diam_{n_1-1}(V_{n_1-1})=2\diam_0(V_0).\]
Furthermore, $V_{n_1}$ is contained in a single interval $I_{n_1}\in\mathcal{I}_{n_1}$. If $\diam_{n_1}(V_{n_1})>\frac{1}{4}\diam_{n_1}(I_{n_1})$, then, via \rf{L:big} and \rf{E:delta_size}, there exists $n\geq N^*$ such that 
\[\diam_n(V_n)\leq4\diam_{n_1}(V_{n_1})\leq 8\diam_0(V_0)\leq 32\delta.\] Therefore, we assume that $\diam_{n_1}(V_{n_1})\leq \frac{1}{4}\diam_{n_1}(I_{n_1})$.

If there is no index $n>n_1$ such that $V_n$ is symmetric about a point in $\mathcal{D}_{n+1}\setminus\mathcal{D}_{n}$, then, by \rf{L:symm_next}, we  conclude that there exists $n\geq N^*$ such that $\diam_n(V_n)\leq 2\diam_0(V_0)\leq8\delta$. Therefore, we may assume that there exists $n_2>n_1$ minimal such that $V_{n_2}$ is symmetric about a point in $\mathcal{D}_{n_2+1}\setminus\mathcal{D}_{n_2}$. Furthermore, the proof of \rf{L:symm_next} makes it clear that $f_{n_2-1}$ is the identity on $I_{n_2-1}$ and so $V_{n_2}=V_{n_2-1}$. Therefore, $V_{n_2}$ is contained in an interval $I_{n_2}\in\mathcal{I}_{n_2}$ and, either $\diam_{n_2}(V_{n_2})=0$, or $\diam_0(V_0)\leq \diam_{n_2}(V_{n_2})\leq2\diam_0(V_0)$. If $\diam_{n_2}(V_{n_2})>\frac{1}{4}\diam_{n_2}(I_{n_2})$, then, via \rf{L:big} and \rf{E:delta_size}, there exists $n\geq N^*$ such that $\diam_n(V_n)\leq 32\delta$. Therefore, we may assume that $\diam_{n_2}(V_{n_2})\leq \frac{1}{4}\diam_{n_2}(I_{n_2})$,

If there is no index $n>n_2$ such that $V_n$ is symmetric about a point in $\mathcal{D}_n$, then, via \rf{L:not_middle} and \rf{E:delta_size}, we conclude that there exists $n\geq N^*$ such that $\diam_n(V_n)\leq 2\diam_0(V_0)\leq 8\delta$. Thus, we assume that there exists $n_3>n_2$ minimal such that $V_{n_3}$ is symmetric about a point in $\mathcal{D}_{n_3}$. 

If $V_{n_3}$ is a single point, then we note that, for all $n\geq n_3$, we have $f_n^{-1}(V_{n_3})=V_{n_3}$ (since $f_n$ fixes points in $\mathcal{D}_{n_3}$). Therefore, there exists $n\geq N^*$ such that $\diam_n(V_n)=0<\delta$. Thus we may assume that $\diam_{n_3}(V_{n_3})>0$. In this case, we note that $n_3$ is minimal such that $V_{n_3}$ is not contained in a single interval from $\mathcal{I}_{n_3}$. It is also easy to verify (via \rf{L:not_middle}) that $\diam_{n_3}(V_{n_3})=\diam_{n_2}(V_{n_2})$. Indeed, $V_{n_3}$ is contained in the interior of the union of two adjacent intervals from $\mathcal{I}_{n_3}$ whose union forms $I_{n_3-1}\in\mathcal{I}_{n_3-1}$. 

Write $I_{n_3}'$ to denote the left dyadic child of $I_{n_3-1}$, and write $V_{n_3}':=V_{n_3}\cap I_{n_3}'$. Inductively, for each $k\geq1$, write $I_{n_3+k}'$ to denote the right dyadic child of $I_{n_3+k-1}'$. Since $\diam_{n_3}(V_{n_3}')\leq\frac{1}{4}\diam_{n_3}(I_{n_3}')$, we have $V_{n_3}'\subset I'_{n_3+2}$. If $\diam_{n_3}(V_{n_3}')=\frac{1}{4}\diam_{n_3}(I_{n_3}')$, then 
\[\diam_{n_3-1}(V_{n_3-1})=2\diam_{n_3}(V_{n_3}')=\frac{1}{2}\diam_{n_3}(I_{n_3}')\geq\frac{1}{4}\diam_{n_3-1}(I_{n_3-1}).\]
Therefore, by \rf{L:big} and \rf{E:delta_size}, there exists $n\geq N^*$ such that 
\[\diam_n(V_n)\leq 4\diam_{n_3-1}(V_{n_3-1})\leq8\diam_0(V_0)\leq32\delta.\]
Therefore, we may assume that $\diam_{n_3}(V_{n_3}')<\frac{1}{4}\diam_{n_3}(I_{n_3}')$.

For each $n\geq n_3$, we define $V_n':=V_n\cap I_{n_3}'$, and we define the ratio
\[R(n):=\frac{\diam_n(V'_n)}{\diam_n(I_n')}.\]
Since $R(n_3)<\frac{1}{4}$, we have $\diam_{n_3+1}(V_{n_3+1}')=\diam_{n_3}(V_{n_3}')$. If $f_{n_3}$ is a folding map on $I_{n_3}'$, then $V_{n_3+1}'\subset I_{n_3+3}'$ and $R(n_3+1)=R(n_3)<\frac{1}{4}$. If $f_{n_3}$ is the identity on $I_{n_3}$, then $R(n_3+1)=2R(n_3)$. If $R(n_3+1)\geq\frac{1}{4}$, then define $n_4:=n_3+1$. If not, then we proceed inductively, and assume that, for all $n_3+1\leq j\leq k-1$, we have $V_{n_3+j}'\subset I_{n_3+j+2}'$, $\diam_{n_3+j}(V_{n_3+j}')=\diam_{n_3}(V_{n_3}')$, and $R(n_3+j)<\frac{1}{4}$. 

Under this inductive hypothesis, we examine $V_{n_3+k}'$. Either $f_{n_3+k-1}$ is a folding map on $I_{n_3+k-1}'$, and $R(n_3+k)=R(n_3+k-1)<\frac{1}{4}$, or $f_{n_3+k-1}$ is the identity on $I_{n_3+k-1}'$, and $R(n_3+k)=2R(n_3+k-1)$. If  $R(n_3+k)\geq\frac{1}{4}$, then write $n_4:=n_3+k$. 

Via induction, we are faced with two possibilities: either there exists $n_4>n_3$ minimal such that $V_{n_4}'\subset I_{n_4}'$ and $R(n_4)\geq\frac{1}{4}$, or, for all $n> n_3$, we have $V_n'\subset I_{n+2}'$ and $R(n)< \frac{1}{4}$. We claim this latter case cannot occur. Indeed, we note that, for all $n>n_3$, we have $R(n+1)\geq R(n)$. Moreover, $R(n+1)>R(n)$ if and only if $f_n$ is the identity on $I_{n}'$ and $R(n+1)=2R(n)$. Since $\diam_n(I_n')\to 0$, the map $f_n$ must be the identity on $I_n'$ infinitely often, and thus $R(n+1)=2R(n)$ infinitely often. This would imply that $R(n)\to+\infty$, and this contradiction proves our claim.

Thus we have $V_{n_4}'=V_{n_4}\cap I_{n_3}'\subset I_{n_4}'$ such that $\diam_{n_4}(V_{n_4}')\geq\frac{1}{4}\diam_{n_4}(I_{n_4}')$. Recall that, for any $n\geq n_4$, the map $f_n$ fixes elements of $\mathcal{I}_{n_4}$ and $\mathcal{I}_{n_3}$. Therefore, for any $n\geq n_4$, we have $V_n'\subset I_{n_4}'$, and so 
\begin{align*}
\diam_n(V_n')&\leq \diam_n(I_{n_4}')=\diam_{n_4}(I_{n_4}')\\
&\leq4\diam_{n_4}(V_{n_4}')=8\diam_{n_3}(V_{n_3})\leq 16\diam_0(V_0)\leq64\delta.
\end{align*}

An analogous argument applies to the set $V_{n_3}'':=V_{n_3}\cap I_{n_3}''$, where $I_{n_3}''$ denotes the right dyadic child of $I_{n_3-1}$. In particular, there exists $n_5>n_3$ such that, if $n\geq n_5$, then $\diam_n(V_n'')\leq 64\delta$. Therefore, there exists $n\geq N^*$ such that
\[\diam_n(V_n)\leq \diam_n(V_n\cap I_{n_3}')+\diam_n(V_n\cap I_{n_3}'')\leq 128\delta.\]

We finish by briefly considering the case that $n_1=0$. If $V_0$ is symmetric about a point in $\mathcal{D}_2\setminus \mathcal{D}_1$, then we argue as in the case that $n_1\geq1$. If $V_0$ is symmetric about a point in $\mathcal{D}_1\setminus\mathcal{D}_0$, then we apply the argument utilized in our above analysis of $V_{n_2}$. If $V_0$ is symmetric about the point in $\mathcal{D}_0$, then we apply (a simple modification of) the argument used in our above analysis of $V_{n_3}$.
\end{proof}

\bibliographystyle{amsalpha}
\bibliography{bib}

\end{document}